\documentclass[12pt, a4paper]{amsart}

\usepackage[english]{babel}
\usepackage{amsmath}
\usepackage{amssymb}
\usepackage{bm}
\usepackage{amscd} 
\usepackage{microtype} 
\usepackage{verbatim} 
\usepackage{color}  
\usepackage{tikz-cd}\usetikzlibrary{babel} 
\usepackage{enumitem}
\usepackage{mathtools} 
\usepackage{cite}
\usepackage{hyperref} 
\usepackage[initials,msc-links,backrefs]{amsrefs}
\usepackage[all]{xy}

\usepackage[OT2, T1]{fontenc}

\title[The profinite point of view]{Infinite groups from the\\ profinite point of view}

 \author[H. Kammeyer]{Holger Kammeyer}
 \author[S. Kionke]{Steffen Kionke}
 
 \address{Heinrich Heine University D{\"u}sseldorf, Faculty of Mathematics and Natural Sciences, Mathematical Institute, Germany}
 \email{holger.kammeyer@hhu.de}

 \address{Faculty of Mathematics and Computer Science, FernUniversit\"at in Hagen, Germany}
 \email{steffen.kionke@fernuni-hagen.de}
 
 \makeatletter
\@namedef{subjclassname@2020}{%
  \textup{2020} Mathematics Subject Classification}
\makeatother

\theoremstyle{plain}
\newtheorem{theorem}{Theorem}
\newtheorem{lemma}[theorem]{Lemma}
\newtheorem{corollary}[theorem]{Corollary}
\newtheorem{proposition}[theorem]{Proposition}

\theoremstyle{definition}
\newtheorem{definition}[theorem]{Definition}
\newtheorem*{definition*}{Definition}

\newtheorem{example}[theorem]{Example}
\newtheorem{question}[theorem]{Question}
\newtheorem{problem}[theorem]{Problem}
\newtheorem*{observation*}{Observation}

\numberwithin{equation}{section}
\numberwithin{theorem}{section}

\providecommand{\ignore}[1]{}

\providecommand{\N}{\mathbb{N}}
\providecommand{\R}{\mathbb{R}}
\providecommand{\Q}{\mathbb{Q}}
\providecommand{\Z}{\mathbb{Z}}

\providecommand{\C}{\mathbb{C}}
\newcommand{\fin}{\trianglelefteq_{\text{f.i.}}}
\newcommand{\ab}{^\text{ab}}
\newcommand{\T}{\mathcal{T}}

\DeclareMathOperator{\Aut}{Aut}
\DeclareMathOperator{\Sym}{Sym}
\DeclareMathOperator{\RiSt}{RiSt}
\DeclareMathOperator{\St}{St}

\newcommand*{\arXiv}[1]{ \href{http://www.arxiv.org/abs/#1}{arXiv:\textbf{#1}}}

\begin{document}

\begin{abstract}
We survey recent work ranging around the question in how far a group, or a property of a group, is determined by the set of finite quotient groups.  Our focus lies on \(S\)-arithmetic groups, branch groups, and their relatives.
\end{abstract}

\maketitle

\section{Introduction}

Finite groups have been studied since the 19th century starting in particular from work of Arthur Cayley. Since then the theory of finite groups developed tremendously and led to the classification of finite simple groups. Today the classification and its consequences provide an invaluable framework to work with finite groups.

For infinite groups one finds quite the contrary: there is an untamable zoo of examples. Nevertheless, some families of infinite groups were carefully studied, in particular groups that are accessible by geometric methods. 

Given this state of affairs, it appears appealing to approach infinite groups using the theory of finite groups. This survey discusses the following guiding question.
\begin{question}\label{question:guiding}
Given an infinite group $\Gamma$. What can be said about $\Gamma$ if one knows all of its finite quotients?
\end{question}
Reformulated in terms of group actions: what do actions of $\Gamma$ on \emph{finite sets} tell us about $\Gamma$?

\medskip

For general groups this question is not very interesting, since there are plenty of infinite groups that do not have any non-trivial finite quotient. An elementary example is the additive group $(\Q,+)$ of rational numbers. Whereas this example is not finitely generated, there are finitely presented infinite groups without non-trivial finite quotients. The first example was given by Higman in 1951 \cite{Higman51}
\[
	\Gamma = \langle a ,b ,c ,d \mid a^{-1}ba=b^2, b^{-1}cb=c^2, c^{-1}dc=d^2, d^{-1}ad=a^2 \rangle.
\]
The right class of groups to study Question \ref{question:guiding} is the class of \emph{residually finite groups}, where every non-identity element can be mapped non-trivially to some finite quotient.
For example all finitely generated abelian groups and all finitely generated linear groups are residually finite. We will see more examples later in this survey.

Let $\Gamma$ be a residually finite group. Then $\Gamma$ embeds as a dense subgroup in the \emph{profinite completion} $\widehat{\Gamma}$. The profinite completion is a compact, totally disconnected Hausdorff topological group, i.e. a \emph{profinite group}. We shall recall some basic facts in the next section.
It was shown in \cite{Dixonetal} that two finitely generated residually finite groups $\Gamma_1$, $\Gamma_2$ have the same set of finite quotients if and only if  $\widehat{\Gamma}_1 \cong \widehat{\Gamma}_2$ as profinite groups. This allows for a reformulation of the guiding question:

\begin{question}
Let $\Gamma$ be a finitely generated, residually finite group. Which properties of $\Gamma$ can be inferred from the profinite completion $\widehat{\Gamma}$?
\end{question}
The class $\mathcal{RF}$ of finitely generated, residually finite groups is the broadest reasonable class to study this question. However, in some cases we will restrict attention to a subclass $\mathcal{C} \subseteq \mathcal{RF}$, e.g. finitely presented groups, arithmetic groups etc.

\begin{definition}
Let $\mathcal{C} \subseteq \mathcal{RF}$ be a class of finitely generated residually finite groups.
A property $P$ of groups is \emph{profinite within $\mathcal{C}$}, if for all $\Gamma_1, \Gamma_2 \in \mathcal{C}$ with $\widehat{\Gamma}_1 \cong \widehat{\Gamma}_2$ either both or none of the groups have $P$.
\end{definition}
For $\mathcal{C} = \mathcal{RF}$ we avoid mentioning the class and simply speak about \emph{profinite properties}. In a nutshell, profinite properties can be read off from the profinite completion (we are not asking for an explicit procedure to decide whether $\Gamma$ has $P$, though). Similarly, a \emph{profinite invariant} is a group invariant, that is determined by the profinite completion for finitely generated residually finite groups.

The purpose of this survey is to give an overview of profinite properties of groups. It turns out that profinite properties are rare. We will first discuss two methods to construct non-isomorphic groups with isomorphic profinite completions that were used to establish non-profiniteness of several properties. Only in the last section will we discuss properties that are actually profinite.
 Our focus lies on properties that are related to group cohomology and $\ell^2$-invariants. The survey will emphasize results that were supported by the DFG priority program "Geometry at Infinity" (DFG 441848266).  We acknowledge additional support from the RTG ``Algebro-Geometric Methods in Algebra, Arithmetic, and Topology'' (DFG 284078965).

\section{Preliminaries and basic concepts}
\subsection{Profinite completion}
For completeness we recall the following
\begin{definition}
A group $\Gamma$ is \emph{residually finite} if for every $\gamma \neq 1_\Gamma \in \Gamma$ there is a finite group $Q$ and a homomorphism $\phi\colon \Gamma \to Q$ with $\phi(\gamma) \neq 1_Q$.
\end{definition}
Put differently, for every element $\gamma \neq 1$ there is a finite index normal subgroup $N \fin \Gamma$ such that $\gamma \not\in N$. Here and throughout, $N \fin \Gamma$ means that $N$ is a normal subgroup of finite index in $\Gamma$. 

Let $\Gamma$ be a residually finite group. The finite quotients of $\Gamma$ correspond to the finite index normal subgroups of $\Gamma$. If $N \subseteq M$ are two normal subgroups of $\Gamma$, then there is a canonical projection $\pi_{N,M}\colon \Gamma/N \to \Gamma/M$. With these projections, the finite quotients of $\Gamma$ form an inverse system of finite groups. The inverse limit
\[
	\widehat{\Gamma} = \varprojlim_{N \fin \Gamma} \Gamma/N
\]
is called the \emph{profinite completion} of $\Gamma$. It is a profinite group when equipped with the inverse limit topology, where the finite quotients carry the discrete topology. The projections $\pi_N \colon \Gamma \to \Gamma/N$ induce a homomorphism $\iota \colon \Gamma \to \widehat{\Gamma}$. It is easy to see that the residual finiteness of $\Gamma$ entails that $\iota$ is injective. We will always identify $\Gamma$ with the image $\iota(\Gamma)$ and think of $\Gamma$ as a subgroup of its profinite completion.
Since the image of $\Gamma$ surjects onto the finite continuous quotients of $\widehat{\Gamma} \to \Gamma/N$ the image $\iota(\Gamma)$ is a dense subgroup of $\widehat{\Gamma}$.

\begin{example} \label{example:profinite-integers}
Let $\Gamma = \Z$. Using the Chinese remainder theorem, one can check that the profinite completion is 
\[
	\widehat{\Z} = \varprojlim_{n} \Z/n\Z \cong \prod_{p \text{ prime}} \Z_p
\]
where $\Z_p$ denotes the additive group of the $p$-adic integers.
\end{example}

The profinite completion has the following universal property:
\begin{lemma}[Universal property]
Let $\Gamma$ be a residually finite group and let $K$ be a profinite group.
For every homomorphism $\phi \colon \Gamma \to K$ there is a unique continuous homomorphism $\hat{\phi} \colon \widehat{\Gamma} \to K$ with $\hat{\phi}|_\Gamma = \phi$.
\end{lemma}
\begin{proof}
The uniqueness is immediate from the density of $\Gamma$ in $\widehat{\Gamma}$.

For the existence, assume first that $K$ is finite. Then the kernel of $\phi$ has finite index in $\Gamma$ and the assertion is clear from the construction of $\widehat{\Gamma}$. Suppose that $K$ is profinite and write $K=\varprojlim_{i\in I} K_i$ as an inverse system of finite groups with projections $\pi_i \colon K \to K_i$. 
We apply the universal property to the homomorphisms $\phi_i = \pi_i \circ \phi \colon \Gamma \to K_i$ to obtain continuous homomorphisms $\hat{\phi}_i \colon \widehat{\Gamma} \to K_i$.
The universal property of the inverse limit gives us a unique continuous homomorphism $\hat{\phi} \colon \widehat{\Gamma} \to K$ with $\pi_i \circ \hat{\phi} = \hat{\phi}_i$. Then $\pi_i \circ \hat{\phi}|_\Gamma = \pi_i \circ \phi$ for all $i \in I$ and this implies $\hat{\phi}|_\Gamma = \phi$.
\end{proof}

Let $\Gamma\ab= \Gamma/[\Gamma,\Gamma]$ denote the abelianisation of $\Gamma$. For a profinite group $G$, the notation $G\ab = G/\overline{[G,G]}$ denotes the profinite abelianisation (i.e., the largest continuous abelian factor group). Using the universal properties, one can check that 
profinite completion and abelianisation commute.
\begin{lemma}\label{lem:profinite-commutes-ab}
Let $\Gamma$ be a residually finite group. There is a canonical isomorphism
\[
	\widehat{\Gamma}\ab \cong \widehat{(\Gamma\ab)}.
\]
\end{lemma}
The open subgroups of $\widehat{\Gamma}$ are in bijective correspondence with the finite index subgroups of $\Gamma$. More precisely~\cite{Ribes-Zalesskii:profinite-groups}*{Proposition~3.2.2}:
\begin{lemma}\label{lem:correspondence}
Let $\Gamma$ be a residually finite group.
\begin{enumerate}
\item If $H \leq \Gamma$ has finite index,  then the closure $\overline{H} \subseteq \widehat{\Gamma}$ is open and $\widehat{H} = \overline{H}$.
\item If $U \leq_o \widehat{\Gamma}$ is an open subgroup, then $H = U \cap \Gamma$ has finite index in $\Gamma$. 
\end{enumerate}
These constructions induce a bijection between the finite index (normal) subgroups of $\Gamma$ and the open (normal) subgroups of $\widehat{\Gamma}$.
\end{lemma}

\subsection{Profinite rigidity}
Here we will briefly discuss the following central question: When is a group uniquely determined up to isomorphism by its profinite completion? To be more concise we start with the following
\begin{definition}
Let $\mathcal{C} \subseteq \mathcal{RF}$ be a class of finitely generated residually finite groups.
A group $\Gamma \in \mathcal{C}$ is \emph{profinitely rigid in $\mathcal{C}$} if all $\Delta \in \mathcal{C}$ with $\widehat{\Delta} \cong \widehat{\Gamma}$ satisfy $\Delta \cong \Gamma$.
\end{definition}
If a finitely generated residually finite group is profinitely rigid in $\mathcal{RF}$, then one says that $\Gamma$ is profinitely rigid in the \emph{absolute sense}.
\begin{example}\label{ex:abelian}
Let $\mathcal{AB}$ be the class of finitely generated abelian groups. Then every $\Gamma \in \mathcal{AB}$ is profinitely rigid in  $\mathcal{AB}$. This is easy to check using the structure theorem for finitely generated abelian groups. Let $\Gamma = \Z^r\times F$ (with $F$ finite abelian), then $\widehat{\Gamma} \cong \widehat{\Z}^r \times F$. The group $F$ is exactly the torsion subgroup of this abelian group. By looking at the number of homomorphisms to $\Z/p\Z$ for a large prime $p$, one can recover the number $r$.
We will see below that finitely generated abelian groups are profinitely rigid in the absolute sense, by showing that being abelian is a profinite property.
\end{example}
Absolute profinite rigidity appears to be a rare phenomenon. It is known since the seventies that already virtually abelian groups \cite{Baumslag74} or finitely generated nilpotent groups \cite{Pickel71} are in general not rigid. 
However, for finitely generated virtually nilpotent groups, there are at most finitely many isomorphism classes of groups that give rise to isomorphic profinite completions \cite{Pickel73}.

Bridson, McReynolds, Reid and Spitler  showed that certain fundamental groups of hyperbolic $3$-manifolds and certain hyperbolic triangle groups are profinitely rigid in the absolute sense \cite{BMRRS20,BMRRS21}.

We have to mention the two major open problems in the field:
\begin{problem}
Are non-abelian free groups profinitely rigid in the absolute sense?
\end{problem}

\begin{problem} \label{problem:slnz-rigid}
Is $\mathrm{SL}_n(\Z)$ for $n \geq 3$ profinitely rigid in the absolute sense?
\end{problem} 

\subsection{Grothendieck pairs}
Every homomorphism $f \colon \Delta \to \Gamma$ between groups induces a continuous homomorphism $\hat{f} \colon \widehat{\Delta} \to \widehat{\Gamma}$ between the profinite completions.
In 1970 Grothendieck published an influential paper \cite{Grothendieck70} where he studied homomorphisms between finitely generated groups  $f \colon \Delta \to \Gamma$ such that $\hat{f}$ is an isomorphism. Among other things, he showed that in this case the categories of representations of $\Delta$ and $\Gamma$ on finitely presented $A$-modules (where $A$ is an arbitrary commutative ring) are equivalent.
He also raised the question \cite[\S 3.1]{Grothendieck70} whether $f$ is an isomorphism provided that $\hat{f}$ is an isomorphism and $\Delta, \Gamma$ are residually finite and \emph{finitely presented}.

If $\Gamma$ and $\Delta$ are residually finite and $\hat{f}$ is injective, then $f$ is injective and hence $f\colon \Delta \to f(\Delta)$ is an isomorphism. This leads to the following definition.
\begin{definition}
A pair $(\Gamma, \Delta)$ consisting of a residually finite, finitely generated group $\Gamma$ and a finitely generated proper subgroup $\Delta \subsetneq \Gamma$ such that the inclusion $\iota \colon \Delta \to \Gamma$ induces an isomorphism $\hat{\iota} \colon \widehat{\Delta} \to \widehat{\Gamma}$ is called a \emph{Grothendieck pair.}
\end{definition}
The existence of Grothendieck pairs was first established by Platonov--Tavgen \cite{PlatonovTavgen} in 1986. The first Grothendieck pair of finitely presented groups, answering the original question of Grothendieck, was constructed by Bridson--Grunewald in 2004 \cite{BridsonGrunewald04}.

\medskip

The notion of Grothendieck pairs allows us to formulate a refined question: What properties are shared by the groups of a Grothendieck pair? Or on the contrary: How different can the groups in a Grothendieck pair be?

\section{Profinitely isomorphic $S$-arithmetic groups}

The class of \(S\)-arithmetic groups provides an illuminating test ground for the profiniteness of a given property or invariant of groups.  For the most hands-on definition of an arithmetic group, we start with a group of matrices \(\mathbf{G} \subseteq \mathbf{GL_n}(\Q)\) defined by the zero locus of a set of rational polynomials in the matrix coefficients. We write \(\mathbf{G}(\Z) = \mathbf{G} \cap \mathbf{GL_n}(\Q)\). Then a subgroup \(\Gamma \le \mathbf{G}\) is called \emph{arithmetic} if \(\Gamma\) and \(\mathbf{G}(\Z)\) are commensurable (i.e., the intersection \(\Gamma \cap \mathbf{G}(\Z)\) has finite index in $\Gamma$ and $\mathbf{G}(\Z)$).  The group \(\Gamma\) comes endowed with the inverse system of \emph{principal congruence subgroups}: the kernels of the coefficient reduction morphisms \(\Gamma \rightarrow \mathbf{GL_n}(\Z / n\Z)\).  If the principal congruence subgroups are cofinal in the system of all finite index normal subgroups of \(\Gamma\), we say that \(\Gamma\) has the \emph{congruence subgroup property (CSP)}.  Notable examples of arithmetic groups enjoying the congruence subgroup property are \(\Gamma = \operatorname{SL_n}(\Z)\) for \(n \ge 3\) and the integral spinor group \(\Gamma = \operatorname{Spin}(q)(\Z)\) of a non-degenerate integral quadratic form \(q\) with Witt index at least two.  The congruence subgroup property has the virtue that the profinite completion has a transparent description.  Indeed, using Example~\ref{example:profinite-integers}, we see that
\[ \widehat{\operatorname{SL}_n(\Z)} \cong \operatorname{SL_n}(\widehat{\Z}) \cong \prod_p \operatorname{SL_n}(\Z_p) \quad \text{and} \quad \widehat{\operatorname{Spin}(q)(\Z)} \cong \prod_p \operatorname{Spin}(q)(\Z_p). \]
This observation can be used to construct profinitely isomorphic arithmetic groups by a trick that was first employed by M.\,Aka~\cite{Aka:kazhdan}.

\begin{lemma}
  For every prime number \(p\), the quadratic form \(\langle 1, 1, 1, 1 \rangle\) is isometric to \(\langle -1, -1, -1, -1 \rangle\) over the \(p\)-adic integers \(\Z_p\).
\end{lemma}

\begin{proof}
  If \(p \equiv 1 \ (4)\), this follows from \(\sqrt{-1} \in \Z_p\) because \(\Z_p\) contains the \((p-1)\)-st roots of unity.  If \(p \equiv 3 \ (4)\), standard quadratic form theory shows that \(\langle 1, 1 \rangle \cong_{\Z_p} \langle -1, -1 \rangle\) and for \(p = 2\), we can use \(\sqrt{-7} \in \Z_2\) to exhibit the explicit isometry
  \[ \begin{pmatrix}
      2 & 1 & 1 & \sqrt{-7} \\
      -1 & 2 & -\sqrt{-7} & 1 \\
      -1 & \sqrt{-7} & 2 & -1 \\
      -\sqrt{-7} & -1 & 1 & 2
      \end{pmatrix}. \qedhere \]
\end{proof}

Hence, if \(q_{n,m}\) denotes the diagonal quadratic form with \(n\) coefficients equal to ``\(1\)'' and \(m\) coefficients equal to ``\(-1\)'', then
\[ \widehat{\operatorname{Spin}(q_{7,2})(\Z)} \cong \widehat{\operatorname{Spin}(q_{3,6})(\Z)}. \]
The groups \(\operatorname{Spin}(q_{7,2})(\Z)\) and \(\operatorname{Spin}(q_{3,6})(\Z)\) are however not isomorphic themselves.  For example, the simple Lie groups in which they are lattices are unique by superrigidity and they have real rank two and three, respectively.  Such examples can be used to show that a number of properties and invariants of groups are not profinite as we will see next.

\subsection{Bounded cohomology} \label{subsection:bounded}
Bounded cohomology \(H^*_b(\Gamma; \R)\) is the variant of real group cohomology that arises from considering only bounded cochains of \(\Gamma\).  An informative account can be found in~\cite{Frigerio:bounded-cohomology}.  By a theorem of Burger--Monod and Monod--Shalom~\cite{Monod-Shalom:cocycle}*{Theorem~1.4}, a lattice \(\Gamma \le G\) in a simple Lie group has \(H^2_b(\Gamma; \R) \neq 0\) if and only if \(G\) defines a hermitian symmetric space.  Since the symmetric space of \(\operatorname{SO}^0(7,2)\) is hermitian while the symmetric space of \(\operatorname{SO}^0(3,6)\) is not, the above example shows that second bounded cohomology is not a profinite invariant.  However, not all hermitian symmetric spaces give rise to counterexamples and the complete list of those who do is given in~\cite{Echtler-Kammeyer:bounded}.

\subsection{Higher \(\ell^2\)-Betti numbers} \label{sec:higherl2}
Instead of bounded (co-)chains, one can also consider square integrable chains resulting in the \(\ell^2\)-homology of groups.  In the \emph{reduced} case, the \(\ell^2\)-homology is uniquely determined by the \emph{von Neumann dimension} which gives rise to the \(\ell^2\)-Betti numbers \(b^{(2)}_n(\Gamma)\) of a group \(\Gamma\).  The precise construction of these invariants can be found in~\cite{Lueck:l2} and~\cite{Kammeyer:l2}.  By classical results of Borel, the group \(\operatorname{Spin}(q_{7,2})(\Z)\) has all \(\ell^2\)-cohomology concentrated in degree seven whereas \(\operatorname{Spin}(q_{3,6})(\Z)\) has \(\ell^2\)-cohomology concentrated in degree nine, so \(\ell^2\)-Betti numbers are not profinite in general.  A closer investigation refining the above construction shows that the two \emph{\(S\)-arithmetic groups}
\[ \textstyle \Gamma^n_{\pm} = \operatorname{Spin}(\langle \pm 1, \pm 1, \pm 1, \pm p_1 \cdots p_n, 3 \rangle)\left(\Z\left[\frac{1}{p_1 \cdots p_n}\right]\right), \]
satisfy \(\widehat{\Gamma^n_+} \cong \widehat{\Gamma^n_-}\) and have non-vanishing \(\ell^2\)-Betti numbers in degree \(n\) and \(n+2\), respectively \cite{Kammeyer-Sauer:spinor}.  Here \(p_1, \ldots, p_n\) with \(n \ge 2\) denote the first \(n\) primes in the arithmetic progression \(89 + 24\mathbb{N}\).  So in fact, all higher \(\ell^2\)-Betti numbers fail to be profinite invariants.  The first \(\ell^2\)-Betti number, however, is well-known to be profinite among finitely presented groups as we recall in Theorem~\ref{thm:first-l2-betti-profinite} below.

\subsection{Kazhdan's property~(T) and higher variants}
In the original paper~\cite{Aka:kazhdan}, M.\,Aka gives
the examples
\[ \operatorname{Spin}(q_{5,2})(\mathbb{Z}[\sqrt{2}]) \quad \text{and} \quad \operatorname{Spin}(q_{1,6})(\mathbb{Z}[\sqrt{2}]) \]
of two profinitely isomorphic groups with and without \emph{Kazhdan's property (T)}.  Recently, U.\,Bader and R.\,Sauer~\cite{Bader-Sauer:higher-kazhdan} have investigated higher versions of property (T), defined in terms of the range of vanishing continuous cohomology with coefficients in unitary represenations.  Since the groups \(\operatorname{Spin}(p,q)\) with even \(p \cdot q\) have discrete series representations, they have nonvanishing continuous cohomology in degree \(pq/2\).  It thus follows easily from the results of Bader--Sauer \cite{Bader-Sauer:higher-kazhdan}*{Theorem~B} that higher Kazhdan properties are not profinite either.

\subsection{Serre's property FA}
In another direction, a group with Kazhdan's property (T) also has \emph{Serre's property FA}: every action on a tree has a global fixed point.  Aka's result has the following strengthening: Property FA is not profinite~\cite{Cheetham-West-et-al:property-fa}.  Fix a prime \(p\) and let \(A\) be the unique quaternion algebra over \(\Q\) that ramifies precisely at \(p\) and at the infinite place. Then both the groups
\[ \textstyle \Gamma = \operatorname{SL}_2(A) \cap \prod_{q \neq p} \operatorname{SL}_4(\Z_q) \quad \text{and} \quad \Delta = \operatorname{SL}_4\left(\Z\left[\frac{1}{p}\right]\right) \]
are \(S\)-arithmetic groups with the congruence subgroup property, so that
\[ \widehat{\Gamma} \cong \prod_{q \neq p} \operatorname{SL}_4(\Z_q) \cong \widehat{\Delta}. \]
But \(\Delta\) is a lattice in the higher rank Lie group \(\operatorname{SL}_4(\R) \times \operatorname{SL}_4(\Q_p)\), so \(\Delta\) has (T), hence FA.  In contrast, the group \(\Gamma\) acts on the ``tree of the group \(\operatorname{SL}_2(\mathbb{H}_p)\)'' without global fixed points.  Here, \(\mathbb{H}_p\) denotes the unique quaternion division algebra over \(\Q_p\) and the construction of the tree is similar to the construction of the Bruhat--Tits tree of \(\operatorname{SL}_2(\Q_p)\).

\subsection{Finiteness properties, center, and torsion elements.} \label{subsection:finiteness}
Finally, A. Lubotzky has given examples of \(S\)-arithmetic groups over function fields showing that the finiteness property \(F_n\) is not visible at all in the profinite completion~\cite{Lubotzky:finiteness-properties}.  A similar construction for number fields in the same article gives that neither having \emph{trivial center} nor being \emph{torsion-free} is a profinite property.

\medskip
All these constructions illustrate that the class \(\mathcal{A}_S\) of

\smallskip
\noindent\emph{\(S\)-arithmetic subgroups \(\Gamma \le \mathbf{G}\) of simply-connected, absolutely almost simple, linear \(k\)-groups \(\mathbf{G}\) over number fields \(k\) with positive \(S\)-rank}
\smallskip

\noindent is utmost flexible in terms of profinite properties.  We may thus ask whether this class of groups exhibits any kind of profinite rigidity at all.  This question seems to be of an entirely different nature in the two cases when \(\mathbf{G}\) has \(S\)-rank one or has higher \(S\)-rank (note that non-abelian free groups occur among the arithmetic subgroups of the rank one \(\Q\)-group \(\mathbf{G} = \mathbf{SL_2}\)).  We shall focus on the higher rank case because a well-known conjecture of Serre says that precisely the \(S\)-arithmetic subgroups \(\Gamma \le \mathbf{G}\) of higher rank groups \(\mathbf{G}\) satisfy the congruence subgroup property in the slightly weaker sense that the \emph{\(S\)-congruence kernel} \(C(\mathbf{G}, S)\) of \(\mathbf{G}\) should be finite.  The latter is defined as follows.  The \(S\)-arithmetic subgroups and the principal \(S\)-congruence subgroups define each a unit neighborhood base of a topology on \(\mathbf{G}(k)\).  Forming the completions with respect to the corresponding uniform structures on \(\mathbf{G}(k)\), we thus obtain the short exact sequence
\[ 1 \rightarrow C(\mathbf{G},S) \longrightarrow \widehat{\mathbf{G}(k)} \longrightarrow \overline{\mathbf{G}(k)} \rightarrow 1 \]
defining \(C(\mathbf{G}, S)\).  The congruence subgroup property in the previous sense means that \(C(\mathbf{G}, S)\) is trivial.  A finite \(S\)-congruece kernel still ensures that we have a good description of the profinite completion.  It is however a notoriously hard problem to decide whether \(\Gamma\) has a Grothendieck subgroup.

\begin{question}[Platonov--Rapinchuk~\cite{Platonov-Rapinchuk:algebraic-groups}*{p.\,434}] \label{question:platonov-rapinchuk}
Let \(\Gamma \le \mathbf{G}\) be an \(S\)-arithmetic subgroup and suppose that \(C(\mathbf{G}, S)\) is finite.  Does there exist \(\Delta \le \Gamma\) such that \((\Gamma, \Delta)\) is a Grothendieck pair?
\end{question}

To the authors' knowledge, the answer to this question is not known for a single such group \(\Gamma\).

\medskip
The possible occurrence of Grothendieck pairs is however not the only impediment to profinite rigidity.  In fact, CSP implies that \(\widehat{\Gamma}\) shares an open subgroup with the locally compact group \(\mathbf{G}(\mathbb{A}^f_k)\) given by the \emph{finite adele points} of \(\mathbf{G}\).  Here \(\mathbb{A}^f_k = \prod'_{v \nmid \infty} k_v\) is defined as the restricted product consisting of those elements in the ordinary product for which almost all coordinates lie in the valuation ring \(\mathcal{O}_v\) of the local completion \(k_v\) of \(k\) at the corresponding non-archimedan place \(v\).  Hence, whenever a \(k\)-group \(\mathbf{G}\) with CSP and an \(l\)-group \(\mathbf{H}\) with CSP satisfy \(\mathbf{G}(\mathbb{A}^f_k) \cong \mathbf{H}(\mathbb{A}^f_l)\), then we can pick an open compact subgroup \(U \le \mathbf{G}(\mathbb{A}^f_k)\)  and the intersections \(\Gamma = U \cap \mathbf{G}(k)\) and \(\Lambda = U \cap \mathbf{H}(l)\) are arithmetic subgroups which have finite index subgroups with isomorphic profinite completions.  By Margulis superrgidity, these are not isomorphic unless \(\mathbf{G}\) is isomorphic to \(\mathbf{H}\) over a field isomorphism \(k \cong l\).  There are two non-trivial ways in which it may happen that \(\mathbf{G}(\mathbb{A}^f_k) \cong \mathbf{H}(\mathbb{A}^f_l)\) and these will be explained in the next two sections.

\subsection{Locally isomorphic number fields}

\begin{definition}
  Two number fields \(k\) and \(l\) are called \emph{arithmetically equivalent} if they have the same Dedekind zeta function \(\zeta_k = \zeta_l\).
\end{definition}

  Equivalently, almost all rational primes have the same \emph{decomposition type} in \(k\) and \(l\), meaning the same unordered tuple of inertia degrees.  It is then actually true that the decomposition types are the same at every (finite) rational prime, so that the local completions of \(k\) and \(l\) can only differ at ramified primes.  If this does not happen, then we have a bijection \(v \mapsto v'\) of the finite places of \(k\) and \(l\) such that \(k_v \cong l_{v'}\).  Equivalently:

\begin{definition}
  We say that \(k\) and \(l\) are \emph{locally isomorphic} if \(\mathbb{A}^f_k \cong \mathbb{A}^f_l\).
\end{definition}
  
Still equivalently, \(k\) and \(l\) are locally isomorphic if \(\mathbb{A}_k \cong \mathbb{A}_l\) holds for the full adele ring \(\mathbb{A}_k = \prod'_v k_v\) because arithmetically equivalent number fields are known to share the same signature~\cite{Klingen:similarities}*{Theorem~III.1.4.(h)}.  Non-isomorphic but arithmetically equivalent number fields arise group theoretically as follows. Let \(G\) be a finite group with non-conjugate subgroups \(H_1\) and \(H_2\) which are however \emph{almost conjugate}: they intersect each conjugacy class of \(G\) in the same number of elements.

\begin{example}
  Let \(G = \operatorname{GL}(3,2)\) be the automorphism group of the Fano plane \(\mathbb{P}^2(\mathbb{F}_2)\) and let \(H_1\) and \(H_2\) be the isotropy groups of some point and some line, respectively.  Then \(H_1\) and \(H_2\) can be chosen to be transposes of one another, hence they have the same permutation character which implies they are almost conjugate.  But the stabilizer group of a line in~\(\mathbb{F}_2^3\) cannot at the same time be the stabilizer group of a hyperplane in~\(\mathbb{F}_2^3\), so \(H_1\) and \(H_2\) are not conjugate.
\end{example}

Whenever the inverse Galois problem is solved for \(G\) so that \(G = \operatorname{Gal}(K/\Q)\) for a Galois extension \(K/\Q\), then the fixed fields of \(H_1\) and \(H_2\) define non-isomorphic arithmetically equivalent number fields \(k\) and \(l\).  Conversely, non-isomorphic arithmetically equivalent number fields \(k\) and \(l\) share the same Galois closure \(K\) and \(\operatorname{Gal}(K/k)\) and \(\operatorname{Gal}(K/l)\) are almost conjugate but not conjugate in \(\operatorname{Gal}(K/\Q)\).  Whether arithmetically equivalent number fields are locally isomorphic must then be checked at the finitely many ramified primes.

\begin{example}
  Continuing the above example, the polynomials
\[ X^7 -7X +3 \quad \text{and} \quad X^7+14X^4-42X^2-21X+9 \]
have the same splitting field with Galois group \(\operatorname{GL}(3,2)\) and they define non-isomorphic arithmetically equivalent number fields~\cite{Trinks:Arithmetisch} corresponding to subgroups \(H_1\) and \(H_2\) as above.  One can check that they also have the same local completions at the ramified primes \(3\) and \(7\), so the fields are locally isomorphic.
\end{example}

If \(\mathbf{G}\) is a \(\Q\)-group with CSP and \(k\) and \(l\) are non-isomorphic locally isomorphic number fields, then \(\mathbf{G}(\mathbb{A}^f_k) \cong \mathbf{G}(\mathbb{A}^f_l)\) and we obtain arithmetic subgroups which are non-isomorphic but profinitely isomorphic as above.  So if we are trying to find profinitely rigid arithmetic groups, it seems advisable to require that the number field \(k\) under consideration satsify the following definition.

\begin{definition}
  We say that a number field \(k\) is \emph{locally determined} if any other number field \(l\) such that \(\mathbb{A}^f_k \cong \mathbb{A}^f_l\) satisfies \(k \cong l\).
\end{definition}

Again, we could have replaced the finite by the full adele rings in the definition.  Being locally determined does not appear to be too restrictive.  For example, any number field of degree \(\le 11\) whose Galois closure has Galois group different from \(\operatorname{GL}(3,2)\), \(\Z / 8\Z \rtimes (\Z / 8\Z)^*\), \(\operatorname{GL}(2,3)\), and \(\operatorname{PSL}_2(11)\) is \emph{arithmetically solitary}, meaning it has no arithmetically equivalent sibling.  This applies in particular to all number fields of degree \(\le 6\).  Moreover, it was shown in~\cite{Chinburg-et-al:geodesics}*{Corollary~1.4} that any \(k\) with precisely one complex place is arithmetically solitary.  We showed that the same is true if \(k\) has precisely one real place and in fact also some other cycle types prevent the occurrence of Gassmann triples~\cite{Kammeyer-Kionke:gassmann}:

\begin{theorem} \label{thm:gassmann}
  If some \(g \in G\) permutes the set \(G/H\) with cycle type
  \begin{enumerate}
  \item \label{item:gassmann-real} \((1,2,\ldots,2)\) or
  \item \((a_1,\ldots,a_r,l)\) with \(l \ge 2\) coprime to \([G : H] \cdot a_1 \cdots a_r\) or
  \item \((a_1,\ldots,a_r,l)\) where \(l\) is prime and coprime to \(a_1 \cdots a_r\) and \(l \notin \{11, \frac{q^k-1}{q-1}\}\) for prime powers q and \(k > 2\),
  \end{enumerate}
  then every subgroup of \(G\) almost conjugate to \(H\) is conjugate to \(H\).
\end{theorem}

We remark that apart from profinite rigiditiy, these results also have applications to spectral rigidity of Riemannian manifolds~\cite{Kammeyer-Kionke:gassmann}*{Theorem~4}.  For more on similarities in the arithmetic of number fields, we recommend Klingen's monograph~\cite{Klingen:similarities}.

\subsection{Locally isomorphic groups}

Even if the number field \(k\) is locally determined, the examples of \(\operatorname{Spin}(q_{7,2})\) and \(\operatorname{Spin}(q_{3,6})\) over \(k=\Q\) show that there exist non-isomorphic \(k\)-groups \(\mathbf{G}\) and \(\mathbf{H}\) with \(\mathbf{G}(\mathbb{A}^f_k) \cong \mathbf{H}(\mathbb{A}^f_k)\).  Of course, this and the previous phenomenon can occur simultaneously so that the following definition is called for.

\begin{definition}
  We say that a \(k\)-group \(\mathbf{G}\) is \emph{locally isomorphic} to an \(l\)-group \(\mathbf{H}\) if there exists an isomorphism \(j \colon \mathbb{A}^f_k \xrightarrow{\cong} \mathbb{A}^f_l\) of \(\mathbb{A}_\Q\)-algebras such that \(\mathbf{G}\) is isomorphic to \(\mathbf{H}\) over \(j\).
\end{definition}

As a consequence of \emph{adelic superrigidity} \cite{Kammeyer-Kionke:adelic}*{Theorem~3.2}, we then have the following result.

\begin{theorem} \label{thm:local-isomorphism}
  Let \(\mathbf{G}\) and \(\mathbf{H}\) be simply-connected, absolutely almost simple, linear algebraic groups over number fields \(k\) and \(l\), respectively.  Suppose both have positive rank and finite congruence kernel with respect to the set of infinite places.  Then the following are equivalent:
  \begin{enumerate}
  \item \label{item:profinite-isomorphism} We have arithmetic groups \(\Gamma \le \mathbf{G}(k)\) and \(\Lambda \le \mathbf{H}(l)\) with \(\widehat{\Gamma} \cong \widehat{\Lambda}\).
  \item \label{item:local-isomorphism} The groups \(\mathbf{G}\) and \(\mathbf{H}\) are locally isomorphic.
  \end{enumerate}
\end{theorem}

\begin{proof}
  \eqref{item:profinite-isomorphism} \(\Rightarrow\) \eqref{item:local-isomorphism}:  This is proven in \cite{Kammeyer-Kionke:adelic}*{Theorem~3.4} where the condition that \(\mathbf{G}\) and \(\mathbf{H}\) should be \emph{algebraically superrigid} follows from CSP. In fact, the constructed local isomorphism ``extends'' the isomorphism \(\widehat{\Gamma} \cong \widehat{\Lambda}\) up to a central character.
  
\eqref{item:local-isomorphism} \(\Rightarrow\) \eqref{item:profinite-isomorphism}:  If \(\mathbf{G}(\mathbb{A}^f_k) \cong \mathbf{H}(\mathbb{A}^f_l)\), pick a compact open subgroup \(U \le \mathbf{G}(\mathbb{A}^f_k)\).  The condition on the rank ensures that \(\prod_{v \mid \infty} \mathbf{G}(k_v)\) is a non-compact Lie group and similarly for \(\mathbf{H}\).  Thus, intersecting \(U\) with \(\mathbf{G}(k)\) and \(\mathbf{H}(l)\) gives infinite congruence subgroups \(\Gamma\) and \(\Lambda\) of \(\mathbf{G}\) and \(\mathbf{H}\), respectively.  Since the congruence kernels of \(\mathbf{G}\) and \(\mathbf{H}\) are finite, we can replace \(\Gamma\) and \(\Lambda\) with finite index subgroups whose profinite completions agree with the closures in \(\mathbf{G}(\mathbb{A}^f_k) \cong \mathbf{H}(\mathbb{A}^f_l)\) and are isomorphic.
\end{proof}
We remark that Theorem~\ref{thm:local-isomorphism} extends immediately to the \(S\)-arithmetic case by \emph{\(S\)-adelic superrigidity}~\cite{KKK:volume}*{Appendix~A}.

\subsection{Profinitely solitary groups}
Recall that by definition, all arithmetic subgroups of a given \(k\)-group \(\mathbf{G}\) are commensurable, so \(\mathbf{G}\) defines a unique commensurability class of arithmetic groups.  Hence the following variant of profinite rigidity is particularly suitable for the study of arithmetic groups.

\begin{definition}~ \label{def:solitude}
  \begin{enumerate}
  \item Two groups \(\Gamma\) and \(\Delta\) are called \emph{profinitely commensurable}, if there exist open subgroups \(U \le \widehat{\Gamma}\) and \(V \le \widehat{\Delta}\) such that \(U \cong V\).
  \item \label{item:solitude} Let \(\mathcal{C} \subseteq \mathcal{RF}\).  We say that \(\Gamma \in \mathcal{C}\) is \emph{profinitely solitary} in \(\mathcal{C}\) if each \(\Delta \in \mathcal{C}\) which is profinitely commensurable with \(\Gamma\) is commensurable with \(\Gamma\).
  \end{enumerate}
\end{definition}

Similarly as before, if \(\mathcal{C} = \mathcal{RF}\) in Definition~\ref{def:solitude}\,\eqref{item:solitude}, then \(\Gamma\) is called \emph{absolutely solitary}.  Clearly also equivalent to \eqref{item:profinite-isomorphism} and \eqref{item:local-isomorphism} in Theorem~\ref{thm:local-isomorphism} is the condition that any two arithmetic subgroups \(\Gamma \le \mathbf{G}(k)\) and \(\Lambda \le \mathbf{H}(l)\) are profinitely commensurable.  So let \(\mathcal{A}\) be the class of \emph{arithmetic groups with CSP} which is defined as the class \(\mathcal{A}_S\) from above where \(S\) is the set of infinite places and where we additionally require the congruence kernel to be finite.  Then Theorem~\ref{thm:local-isomorphism} implies that \(\Gamma \le \mathbf{G}(k)\) is profinitely solitary in \(\mathcal{A}\) if and only if \(k\) is locally determined and \(\mathbf{G}\) is determined up to \(k\)-isomorphism by the local isomorphism type.  In the special case that \(\mathbf{G}\) is \(k\)-split, we can rephrase the latter in Galois cohomological terms by saying that the diagonal map
\[ H^1(k, \operatorname{Aut} \mathbf{G}) \longrightarrow \prod_{v \nmid \infty} H^1(k_v, \operatorname{Aut} \mathbf{G}) \]
has trivial kernel.  In joint work of the first author with R.\,Spitler~\cite{Kammeyer-Spitler}*{Theorem~2}, we identified all cases in which this happens so that we have the following result.

\begin{theorem}
  Let \(\Gamma \le \mathbf{G}(k)\) be an arithmetic group from the class \(\mathcal{A}\) such that \(\mathbf{G}\) is \(k\)-split.  Then \(\Gamma\) is profinitely solitary in \(\mathcal{A}\) if and only if \(k\) is locally determined and one of the following is true:
\begin{enumerate}
\item \(k\) is totally imaginary,
\item \(\mathbf{G}\) has type \(A_{2n}\),
\item \label{item:one-real-place} \(k\) has exactly one real place and \(\mathbf{G}\) has type \(A_{2n+1}\) or \(C_n\).
\end{enumerate}
\end{theorem}

Note that the ``locally determined''-condition is redundant in case~\eqref{item:one-real-place} due to Theorem~\ref{thm:gassmann}\,\eqref{item:gassmann-real}.  Of course, we would like to promote this theorem to the absolute case.  A huge step in this direction is provided in Spitler's thesis~\cite{Spitler:profinite}*{Theorem~7.1}, where he shows the following.  Let \(\Delta\) be any finitely generated residually finite group such that \(\widehat{\Delta} \cong \widehat{\Gamma}\) for some higher rank arithmetic group \(\Gamma \le \mathbf{G}(k)\) from \(\mathcal{A}\).  Then \(\Delta\) embeds as a subgroup \(\Delta \le \Lambda\) of some possibly distinct arithmetic group \(\Lambda \le \mathbf{H}(l)\) such that \(\mathbf{G}\) is locally isomorphic to \(\mathbf{H}\).  If moreover \(\mathbf{H}\) has CSP, then \(\widehat{\Gamma} \cong \widehat{\Lambda}\) so that either \(\Delta = \Lambda\) or \((\Lambda, \Delta)\) is a Grothendieck pair.

This result allows us to decide for every arithmetic subgroup of a \(k\)-split algebraic group (\emph{Chevalley groups}) whether it is absolutely solitary, except for the uncertainties due to the incomplete status of Serre's conjecture on CSP and, more importantly, Question~\ref{question:platonov-rapinchuk}.  The precise statements are given in \cite{Kammeyer-Spitler}*{Theorems~1 and~3}.  Instead of reproducing them here, we just give a few examples illustrating that the naked eye is unable to judge whether a given Chevalley group is solitary:

\begin{example} \cite{Kammeyer-Spitler}*{Example in Section~1}
  These Chevalley groups either contain a Grothendieck subgroup or are profinitely solitary:
  \[ \operatorname{SL}_2(\Z[ \sqrt{2} ]), \operatorname{SL}_n(\Z) \,\text{for}\, n \ge 3, \operatorname{SL}_3( \Z[\sqrt{7} ]), \operatorname{Spin}(4,5)( \Z ), \operatorname{Sp}_n( \Z ), G_2(\Z). \]
  These Chevalley groups are not profinitely solitary:
  \[ \operatorname{SL}_3( \Z[\sqrt[8]{7} ]), \operatorname{Spin}(5,6)( \Z ), \operatorname{Sp}_n( \Z[\sqrt{2}] ), E_6(\Z), E_7(\Z), E_8(\Z), G_2( \Z[\sqrt{2}] ). \]
\end{example}
In ongoing joint work of the first author with A.\,Baumann, we want to extend such results from Chevalley groups to arithmetic subgroups of general simple algebraic \(k\)-groups.

About the relationship between profinite solitude and profinite rigidity, let us point the reader's attention to work of Weiss Behar~\cite{Weiss:non-rigidity}.  He gives three different methods to show that except for the exceptional types \(E_8\), \(F_4\), and \(G_2\), one can always find two finite index subgroups of a higher rank arithmetic group with CSP such that these subgroups are not isomorphic but have isomorphic profinite completions.  This applies in particular to the Chevalley groups \(\operatorname{SL}_n(\Z)\) for \(n \ge 3\) which one might suspect to be both absolutely profinitely rigid (Problem~\ref{problem:slnz-rigid}) and absolutely profinitely solitary. 

\subsection{Lattices in simple Lie groups}

As we mentioned, Bridson--McReynolds--Reid--Spitler have constructed absolutely profinitely rigid lattices in the simple Lie groups \(\operatorname{PSL}_2 (\C)\) and \(\operatorname{PSL}_2 (\R)\) in~\cite{BMRRS20,BMRRS21} and conjecturally, all such lattices are absolutely profinitely rigid.  So one might wonder whether there are more simple Lie groups \(G\) in which all lattices are absolutely profinitely rigid.  If \(\operatorname{rank} G \ge 2\), one can however employ CSP and the arithmetic machinery from above in a straightforward way to see that some lattices in \(G\) are not even profinitely rigid among themselves.

\begin{theorem} \cite{Kammeyer-Kionke:lattices}*{Theorem~1.3}
  Let \(G\) be a connected simple Lie group.  Suppose \(G\) has higher rank, trivial center, and is neither isomorphic to \(\operatorname{PSL}_m(\mathbb{H})\) nor to any real or complex form of type \(E_6\).  Then for each \(n \ge 2\), there exist \(n\) cocompact lattices in \(G\) which are pairwise non-isomorphic but all have isomorphic profinite completions.
\end{theorem}

The silly exceptions of \(\operatorname{PSL}_m(\mathbb{H})\) and type \(E_6\) groups are only due to the incomplete status of Serre's conjecture on CSP.  So presumably, they can be dropped.  The lattices are constructed as arithmetic congruence subgroups of different levels in a \(k\)-group which is isomorpic to \(G\) at one infinite place and anisotropic at all other infinite places of \(k\).  Therefore these lattices, while pairwise non-isomorphic, are all commensurable.  That is why it seems to be more interesting in the higher rank case to ask whether lattices in \(G\) are profinitely solitary among themselves.  Regarding this question, we prove the following.

\begin{theorem} \cite{Kammeyer-Kionke:lattices}*{Theorem~1.1}
  Let \(G\) be a connected simple Lie group.  Suppose \(G\) has higher rank, trivial center, and is neither
  \begin{itemize}
  \item complex of type \(E_8\), \(F_4\), \(G_2\),
  \item nor real or complex of type \(E_6\),
  \item nor isomorphic to \(\operatorname{PSL}_m(\mathbb{H})\), \(\operatorname{SL}_{2m+1}(\R)\), \(\operatorname{PSL}_{2m+1}(\C)\).
  \end{itemize}
  Then for each \(n \ge 2\), there exist \(n\) cocompact lattices in \(G\) that are pairwise non-commensurable but all have isomorphic profinite completions.
\end{theorem}

The exceptions in the last two bullet points can again be dropped if Serre's conjecture on CSP is true.  In contrast and similarly as in Weiss-Behar's theorem~\cite{Weiss:non-rigidity} mentioned above, the complex groups of type \(E_8\), \(F_4\), \(G_2\) are true exceptions.

\begin{theorem} \cite{Kammeyer-Kionke:lattices}*{Theorem~1.2}
  Let \(G\) be a connected complex simple Lie group of exceptional type \(E_8\), \(F_4\), or \(G_2\).  Then lattices in \(G\) are profinitely solitary among themselves.
\end{theorem}

Again, we may ask if these lattices are in fact absolutely solitary.  The answer depends on whether the lattice is cocompact and on Question~\ref{question:platonov-rapinchuk}.

\begin{theorem} \cite{Kammeyer:absolutely}*{Theorem~1} \label{thm:exceptional}
  Let \(G\) be a connected complex simple Lie group of exceptional type \(E_8\), \(F_4\), or \(G_2\) and let \(\Gamma \le G\) be a lattice.
  \begin{itemize}
  \item If \(\Gamma\) is cocompact, then \(\Gamma\) is not absolutely solitary.
  \item If \(\Gamma\) is non-cocompact, then \(\Gamma\) either contains a Grothendieck subgroup or is absolutely solitary.
  \end{itemize}
\end{theorem}

The paper~\cite{Kammeyer:absolutely} contains a couple of more results in this direction.  For example non-cocompact lattices in \(\operatorname{Sp}_{2n}(\R)\) exhibit the same dichotomy as the non-cocompact lattices in Theorem~\ref{thm:exceptional}, see~\cite{Kammeyer:absolutely}*{Theorem~4}.

\subsection{Profinite properties among \(S\)-arithmetic groups}

In Sections~\ref{subsection:bounded}--\ref{subsection:finiteness}, we have listed various properties which can differ for \(S\)-arithmetic groups with isomorphic profinite completions.  It should thus be worthwhile to point out two invariants which \emph{are} profinite among \(S\)-arithmetic groups: the sign of the Euler characteristic and the volume.

\medskip
\noindent \emph{Sign of the Euler characteristic}.  Recall from Section~\ref{sec:higherl2} that there exist \(S\)-arithmetic groups \(\Gamma^n_\pm\) such that \(\widehat{\Gamma^n_+} \cong \widehat{\Gamma^n_-}\) and such that \(b^{(2)}_k(\Gamma^n_+) > 0\) if and only if \(k = n\) while \(b^{(2)}_k(\Gamma^n_-) > 0\) if and only if \(k = n+2\).  One can alter the construction of such groups in one way or another, but it always seems to occur that the degree of non-vanishing \(\ell^2\)-cohomology has the same parity for profinitely isomorphic \(S\)-arithmetic spinor groups.  This raises the question whether there exists a general obstruction which would explain the observation.  The \emph{Euler characteristic} comes to mind as it is not only the alternating sum of Betti numbers but also of \(\ell^2\)-Betti numbers, \(\chi(\Gamma) = \sum_{n \ge 0} (-1)^n b^{(2)}_n(\Gamma)\), and because \(S\)-arithmetic groups have non-vanishing \(\ell^2\)-homology in at most one degree.  Indeed, first in the arithmetic case~\cite{KKRS:profinite-invariants}, then in the \(S\)-arithmetic case~\cite{Kammeyer-Serafini:euler}, the following result was proven.

\begin{theorem} \label{thm:sign}
  Given a number field \(k\) and a finite set of places~\(S\) including all archimedean places, let \(\Gamma_i \le \mathbf{G_i}\) for \(i = 1,2\) be two \mbox{\(S\)-arithmetic} subgroups of simply-connected simple \(k\)-groups with CSP.  If \(\Gamma_1\) is profinitely commensurable with \(\Gamma_2\), then \(\operatorname{sgn} \chi(\Gamma_1) = \operatorname{sgn} \chi(\Gamma_2)\).
\end{theorem}

Here \(\operatorname{sgn} x\) takes the values \(-1, 0, 1\) if \(x < 0, x = 0, x > 0\), respectively.  Note that in contrast, the absolute value \(|\chi(\Gamma)|\) is not even a profinte invariant for arithmetic groups.  Indeed, if \(\Gamma_{m,n}\) denotes the level-four principal congruence subgroup of \(\operatorname{Spin}(q_{m,n})(\Z)\), then
\[ \chi(\Gamma_{8,2}) = 2^{89} \cdot 5^2 \cdot 17 \quad \text{while} \quad \chi(\Gamma_{4,6}) = 2^{90} \cdot 5^2 \cdot 17 \]
as was calculated in~\cite{KKRS:profinite-invariants}*{Theorem~1.2}.  Moreover, the sign of the Euler characteristic is not profinite within \(\mathcal{RF}\), in fact not even within residually finite groups of finiteness type \(F\)~\cite{KKRS:profinite-invariants}*{Corollary~1.3}.

We remark that the proof of Theorem~\ref{thm:sign} in the arithmetic case does not seem to generalize easily to the \(S\)-arithmetic case.  For arithmetic groups, the sign of the Euler characteristic (if nonzero) can be identified with the parity of half the dimension of the symmetric space.  The latter can be inferred form the Killing form by means of the Weil product formula for quadratic forms.  In the \(S\)-arithmetic case, one needs to add the ranks mod two of \(\mathbf{G}\) at the non-archimedean places in \(S\).  So the interaction of archimedean and non-archimedean places needs to be addressed.  This is done case by case by means of Poitou--Tate duality for the Galois cohomology of the center of \(\mathbf{G}\).

\medskip
\noindent \emph{Volume of \(S\)-arithmetic groups}.  While the above shows that the Euler characteristic of profinitely isomorphic lattices in different Lie groups can differ, the Euler characteristic is a profinite invariant for lattices with CSP in the same Lie group.  This follows from the following more general result because the Euler characteristic is proportional to the covolume.

\begin{theorem} \cite{KKK:volume}*{Theorem~1.2} \label{thm:volume}
  Let \(\Gamma, \Lambda \le G\) be irreducible lattices with CSP* in a connected, higher rank, semisimple Lie group without compact factors and with finite center and let \(\mu\) be a fixed Haar measure. Then \(\mu(G/\Gamma) = \mu(G/\Lambda)\).
\end{theorem}

Here, the acronym CSP* refers to a finite congruence kernel and two additional technical requirements that are conjecturally always granted~\cite{KKK:volume}*{Definition~2.3}.  Many (it should be fair to say ``most'') higher rank lattices are known to have CSP*.  In particular, Theorem~\ref{thm:volume} holds unconditionally if \(\Gamma\) and \(\Lambda\) are non-cocompact~\cite{KKK:volume}*{Theorem~1.3}.  While the congruence subgroup property is essential in our arguments, we do also get a result on a rank one group \(G\) if we restrict attention to arithmetic congruence lattices.  To present the theorem, recall that an \emph{octonionic hyperbolic manifold} \(M\) is a 16-dimensional connected Riemannian manifold whose universal covering is isometric to \(\mathbb{OH}^2\), the \emph{octonionic hyperbolic plane}.  By the deck transformation action, the fundamental group \(\pi_1 M\) embeds (uniquely up to conjugation) into the exceptional Lie group \(\operatorname{Isom}^+(\mathbb{OH}^2) = F_{4(-20)}\).

\begin{theorem}
  Let \(\Gamma = \pi_1 M\) and \(\Lambda = \pi_1 N\) where \(M\) and \(N\) are octonionic hyperbolic manifolds with finite volume.  Suppose that both \(\Gamma\) and \(\Lambda\) define arithmetic congruence lattices in \(F_{4(-20)}\) and that \(\widehat{\Gamma} \cong \widehat{\Lambda}\).  Then \(M\) and \(N\) have equal volume.
\end{theorem}

The theorem might be taken as evidence for an affirmative answer to the following long standing open problem which was most prominently asked in~\cite{Reid:ICM}*{Question~7.4}.

\begin{question} \label{question:volume}
  Let \(\Gamma = \pi_1 M\) and \(\Lambda = \pi_1 N\) where \(M\) and \(N\) are hyperbolic 3-manifolds with finite volume.  Suppose that \(\widehat{\Gamma} \cong \widehat{\Lambda}\).  Do \(M\) and \(N\) have equal volume?
\end{question}

Recall that only finitely many different hyperbolic 3-manifolds can have the same volume.  So if the answer to Question~\ref{question:volume} is yes, then only finitely many hyperbolic 3-manifolds can have profinitely isomorphic fundamental groups.  This corollary was recently obtained by Yi Liu~\cite{Liu:almost} using a blend of 3-manifold methods.  In Section~\ref{section:profinite-properties}, we will finally list some properties of groups that are profinite in the absolute sense.

\section{Profinitely isomorphic branch groups and their relatives}
We have seen in the previous section that the congruence subgroup property gives rise to a concise description of the profinite completion of arithmetic groups and that this description can be used to find non-isomorphic groups with isomorphic profinite completions. A similar strategy works for \emph{branch groups} and some groups inspired by branch group methods. We will begin by describing the class of branch groups.

\subsection{Branch groups}
The notion of branch groups grew out of the famous groups constructed by Grigorchuk in 1980 \cite{Grigorchuk80}. We refer to \cite{Bartholdi-Grigorchuk-Sunic} for an introduction to branch groups with many examples. 
 
 \medskip
 
Let $X = (X_\ell)_{\ell \in \N}$ be a sequence of finite sets with $|X_\ell| \geq 2$. For each $\ell \geq 0$, we define the  $X^{\ell} = X_1 \times \ldots \times X_{\ell}$. We think of the elements of $X^\ell$ as \emph{words of length $\ell$} where the $i$th letter is taken from $X_i$. We also put $X^{0} = \{\emptyset\}$ to consist of the empty word $\emptyset$ of length $0$.
With $X$  we associate a rooted tree $\T = \T_X$. The tree $\T$ has vertex set $X^{\ast} = \bigcup \limits_{\ell=0}^{\infty} X^{\ell}$ and the root vertex is $\emptyset$. An edge connects two vertices  $v \in X^{\ell-1},w \in X^\ell$ if and only if there is $x \in X_\ell$ with $w =vx$. The tree $\T$ is spherically homogeneous: the degree of a vertex depends only on the distance from the root.

\medskip

Let $\Aut(\T)$ be the group of all root-preserving automorphisms of~$\T$.
The distance of a vertex $v$ to the root (i.e. the length of the word $v$ -- called the \emph{level} of $v$), is preserved under the action of $\Aut(\T)$.
In other words, the level sets $X^{\ell}$ are stable under the $\Aut(\T)$-action. Considering only the action on $X^\ell$ provides a homomorphism $\pi_{\ell} \colon \Aut(\T) \rightarrow \Sym(X^{\ell})$. The coarsest topology such that these homomorphisms are continuous gives $\Aut(\T)$ the structure of a profinite group.

\medskip

A subgroup $\Gamma \subseteq \Aut(\T)$ acts \emph{spherically transitively} on $\T$ if the action of $\Gamma$ on $X^{\ell}$ is transitive for every $\ell \geq 1$.
The $\ell$th \emph{level stabilizer} in $\Gamma$ is the kernel of $\pi_\ell$. More precisely, if $\St_\Gamma(v)$ denotes the stabilizer of $v$ in $\Gamma$, then $
\St_\Gamma(\ell) = \bigcap \limits_{v \in X^{\ell}} \St_\Gamma(v)$.

The subgroup of $\St_\Gamma(v)$ consisting of all elements that fix every word that does \emph{not} contain $v$ as an initial subword, is called the \emph{rigid stabilizer} of $v$ in $\Gamma$. We denote it by $\RiSt_\Gamma(v)$. The \emph{$\ell$-th rigid level stabilizer} $\RiSt_\Gamma(\ell)$ is 
the subgroup generated by the groups $\RiSt_\Gamma(v)$ with $v \in X^{\ell}$. We note that $\RiSt_\Gamma(\ell)$ is a normal subgroup of $\Gamma$.
\begin{definition}
Let $\Gamma$ be a subgroup of $\Aut(\T)$.
We say that $\Gamma$ is a \emph{branch subgroup} of $\Aut(\T)$ if $\Gamma$ acts spherically transitively on $\T$ and
 $\RiSt_\Gamma(\ell)$ has finite index in $\Gamma$ for every $\ell \in \N$.
 \end{definition}
 
 For instance $\Aut(\T)$ is a branch subgroup, since in this case we have $\St_{\Aut(\T)}(\ell) = \RiSt_{\Aut(\T)}(\ell)$. Every element $g$ acting trivially on the $\ell$-th level, can be decomposed as product of automorphisms $g = \prod_{v \in X^{\ell}} g_v$ where $g_v$ acts only on words that start with $v$.
 
 \medskip
 
Among branch groups, in particular the Grigorchuk--Gupta--Sidki groups received a lot of attention; see \cite[\S 2.3]{Bartholdi-Grigorchuk-Sunic}.
 For our purposes the most important examples are branch groups inspired by a construction of P.~M.~Neumann \cite{Neumann86}.
 \begin{example}
 Assume that $X_i = \{1,2,3,4,5\}$ for all $i$, i.e., $\T$ is a rooted $5$-regular tree. We define two actions of the alternating group $\mathrm{Alt}(5)$ on $\T$. The first copy of $\mathrm{Alt}(5)$ simply acts be permuting the first letter of every word, i.e. for $\sigma \in \mathrm{Alt}(5)$ we define $\underline{\sigma}(xw) = \sigma(x)w$. The second action is along the ``spine'' $(5,5,5,5,5...)$ of the tree. For $\sigma$ in $\mathrm{Alt}(5)$ we define recursively
  \[
  	\tilde{\sigma}(xw) = \begin{cases} x\underline{\sigma}(w) & \text{ if } x=1\\
	x\tilde{\sigma}(w) &\text{ if } x=5\\
	xw & \text{ otherwise } \end{cases}.
  \]
  This provides us with two copies $\underline{\mathrm{Alt}(5)}, \widetilde{\mathrm{Alt}(5)}$ of the alternating group in $\Aut(\T)$. The group $\Gamma = \langle \underline{\mathrm{Alt}(5)}, \widetilde{\mathrm{Alt}(5)} \rangle$ is a perfect branch group where level and rigid stabilizers agree. In fact, $\Gamma$ satisfies $\Gamma \cong \Gamma \wr \mathrm{Alt}(5)$ and this makes it possible to describe the profinite completion of $\Gamma$ as an iterated wreath product, i.e.
  $\widehat{\Gamma} = \cdots \mathrm{Alt}(5) \wr \mathrm{Alt}(5) \wr \mathrm{Alt}(5)$.
 \end{example}
 The construction of Neumann allows for a lot of flexibility. One can vary the spine, use multiple spines and change the group used for the spinal-action; we refer to section 3.2 in \cite{KS-amenable} for details. Generalizing this further, it was proven in \cite{KS-embedding} that every residually finite group $\Gamma$ embeds into a perfect branch group $B_\Gamma$ preserving a number of properties. For instance, if $\Gamma$ is amenable, then $B_\Gamma$ is amenable, if $\Gamma$ is torsion, then $B_\Gamma$ is torsion.
 \begin{definition}
A branch subgroup $\Gamma$  of $\Aut(\T)$ has the \emph{congruence subgroup property} if the canonical homomorphism $\widehat{\Gamma} \to \Aut(\T)$ is injective.
 \end{definition}
 If $\Gamma$ has the congruence subgroup property, then the closure $\overline{\Gamma}$ in $\Aut(\T)$ agrees with the profinite completion of $\Gamma$. In general, the kernel of $\widehat{\Gamma} \to \Aut(\T)$ is called the congruence kernel and has been studied for instance in \cite{Bartholdi-etal12,Garrido16}. For perfect branch groups à la Neumann the congruence subgroup property can be easily checked using the following
 \begin{lemma}[Corollary 4.4 in \cite{KS-amenable}]
Let $\Gamma \leq \Aut(\T)$ be a spherically transitive subgroup.
Suppose that the rigid vertex stabilizers are perfect and that $\St_\Gamma(\ell) =\RiSt_\Gamma(\ell)$ for every $\ell \geq 1$.
Then $\Gamma$ is just-infinite and satisfies the congruence subgroup property.
\end{lemma}
\begin{proof}
Let $N$ be a non-trivial normal subgroup of $\Gamma$. We need to show that $N$ is a congruence subgroup, i.e. that it contains $\St_\Gamma(n)$ for some $n$.
It is well-known that there is some $n$ such that the derived subgroup of $\RiSt_{\Gamma}(n)$ is contained in $N$; see \cite[Lemma 4]{Segal01} and \cite[Theorem $4$]{Grigorchuk00}.
Since the rigid stabilizer is generated by the rigid vertex stabilizers, it is perfect and we deduce $\St_{\Gamma}(n) = \RiSt_{\Gamma}(n) \subseteq N$.
\end{proof}
Using this, one can describe for branch groups à la Neumann and their relatives the profinite completion as an iterated wreath product.  In the special construction used in \cite{KS-amenable} the profinite completion only depends on the action on the first level \cite[Corollary 4.6]{KS-amenable}. This allows for the construction of uncountable families of profinitely isomorphic but non-isomorphic branch groups. The spinal action allows one to embed any residually finite group as a subgroup. On the other hand, amenability of branch groups can be studied using the action on the boundary of the tree based on the methods of \cite{JNdlS} (one can also use the more specific result \cite{BKN10}). This leads to the main result of \cite{KS-amenable}
\begin{theorem}[{\cite[Theorem 1.1]{KS-amenable}}]
There are $2^{\aleph_0}$ many Grothendieck pairs $(\Gamma,A)$ where $A$ is amenable and $\Gamma$ contains a non-abelian free group. In particular, amenability is not a profinite property.
\end{theorem}
The result in \cite{KS-amenable} only claims that there are uncountably many Grothendieck pairs, but the proof shows  that the cardinality is~$2^{\aleph_0}$.
Moreover, one may take a single finitely generated amenable group $A$ as part of $2^{\aleph_0}$ Grothendieck pairs. The method is rather flexible and should be useful in constructing Grothendieck pairs $(\Gamma,\Delta)$ of groups where $\Gamma$ contains a prescribed finitely generated residually finite group $H$ and $\Delta$ does not contain a subgroup isomorphic to $H$. We expect that containing a fixed (non-trivial) finitely generated residually finite group as a subgroup in general is not a profinite property.

\subsection{Beyond branch groups}
 The profinite completion of a branch group of Neumann type is an iterated wreath product, i.e. a special case of an iterated semidirect product. Recent work \cite{KS-telescopes, KS-hereditarily} suggests that ideas and constructions that give rise to branch groups with a prescribed profinite completion might work in a far more general setting. 
 
In an imprecise sense, it seems that a finitely generated profinite group $G = \varprojlim G_j$ that can be described as an iterated semidirect product of finite groups -- i.e. $G_{j+1} = K_j \rtimes G_{j}$ -- can often be realized to be a profinite completion of a group with generators defined recursively.
Beyond the branch group case, this has been successfully carried out in two cases.

\subsubsection{Direct products and telescopes}
Consider an infinite direct product $G = \prod_{i} \mathrm{Alt}(n_i)$ of alternating groups with $\dots n_i<n_{i+1}<\dots$.
It was proven by Kassabov-Nikolov \cite{KassabovNikolov} that these profinite groups are profinite completions of finitely generated residually finite groups. However, their method produced groups that required  $\geq 10$ generators. Since the profinite group $G$ is two generated, they wondered whether $2$-generated groups exist whose profinite completion is an infinite product of alternating groups.
Transferring ideas from the theory of branch groups to this setting led  to the notion of a telescope of groups  \cite{KS-telescopes} that allowed to prove
\begin{theorem}[{\cite[Theorem 1.4]{KS-telescopes}}]
For every $d \geq 5$
there is a residually finite $2$-generated group $\Gamma$ with $\widehat{\Gamma} \cong \prod_{n=1}^\infty
\mathrm{Alt}(d^n)$.
\end{theorem}
The description of these groups is rather concrete and admits some flexibility. Varying these groups one can find Grothendieck pairs of groups with profinite completion $\prod_{i} \mathrm{Alt}(n_i)$. Building on work of Kassabov and Ershov, Jaikin-Zapirain, Kassabov \cite{Ershov-Jaikin-Kassabov} allows to prove the following
\begin{theorem}[{\cite[Theorem 1.5]{KS-telescopes}}]
There is a Grothendieck pair $(\Gamma,A)$, where $A$ is amenable and $\Gamma$ has property ($\tau$).
\end{theorem}
Property $(\tau)$ is a weakening of Kazhdan's property $(T)$: the trivial representation is isolated among all irreducible representations with finite image; see \cite{Lubotzky-tau}.
\begin{question}
Is there a Grothendieck pair $(\Gamma,A)$ where $A$ is amenable and $\Gamma$ has property (T)?
\end{question} 

\subsubsection{Hereditarily just-infinite groups with positive first $\ell^2$-Betti numbers}
Let $\Pi = (p_i)_{i \in \N}$ be a sequence of prime numbers. A group $\Gamma$ is said to be \emph{$\Pi$-graded}, if there is a descending chain of finite index normal subgroups $\Gamma = N_0 \trianglerighteq N_1 \trianglerighteq N_2 \dots$ of $\Gamma$ with 
$\bigcap_{i\in \N} N_i = \{1\}$ such that $N_{i-1}/N_{i}$ is an elementary $p_i$-group. 
In general, being $\Pi$-graded for some $\Pi$ is not a strong condition. However, for torsion groups a $\Pi$-grading can be very useful; this was already observed by Osin in his construction of finitely generated residually finite torsion groups with positive rank gradient \cite{Osin2011}.
It was shown in \cite[Proposition 2.3]{KS-hereditarily} that if $\Pi$ is a sequence of pairwise distinct primes and $\Gamma$ is a $\Pi$-graded \emph{torsion group}, then the profinite completion of $\Gamma$ is determined by the grading:
\[
	\widehat{\Gamma} = \varprojlim_{i \in \N} \Gamma/N_i.
\]
This implies that the profinite completion in this case is an iterated semi-direct product of finite groups $G_{j+1} = V_j \rtimes G_{j}$ where $V_j$ is an elementary abelian $p_j$-group and a $G_j$-module.
More surprisingly, every quotient of $\Gamma$ is $\Pi$-graded and hence residually finite.

This allows to take another approach in the construction of $\Pi$-graded torsion groups. As a first step one constructs the profinite group $G$ as an iterated semi-direct product of elementary abelian groups. In the second step one modifies the involved modules $V_j$ to find a dense torsion subgroup $\Gamma$. The group $G$ will automatically be the profinite completion of $\Gamma$. This approach allows to control further properties of the group $\Gamma$. In particular, Eduard Schesler and the second author proved
\begin{theorem}[\cite{KS-hereditarily}]\label{thm:hji}
Let $\varepsilon > 0$ and let $d\geq 2$.
For every sequence $\Pi =(p_i)_{i\in \N}$ of pairwise distinct primes, there is a $d$-generated, hereditarily just-infinite, $\Pi$-graded torsion group $\Gamma$
with first $\ell^2$-Betti number at least $d-1-\varepsilon$.
\end{theorem}
A group $\Gamma$ is \emph{just-infinite} if every non-trivial normal subgroup has finite index. It is \emph{hereditarily just-infinite}, if every finite index subgroup is just-infinite.  
To our knowledge virtually abelian groups and centerless $S$-arithmetic groups in higher rank simple groups were the only known examples of hereditarily just-infinite (discrete) groups known before \cite{KS-hereditarily} appeared.

The profinite completion of $\Gamma$ in the theorem involves infinitely many primes and these groups are not residually-$p$ for a single prime. In fact, this is necessary. It was shown in \cite{Jaikin-Kionke} that a finitely generated just-infinite group that is residually-$p$ for some prime $p$ has a vanishing first $\ell^2$-Betti number. It would be interesting to know if there are hereditarily just-infinite groups with positive first $\ell^2$-Betti number whose profinite completion involves only finitely many primes.

\section{Profinite properties} \label{section:profinite-properties}
In this last section we finally discuss some properties that are known to be profinite.

\subsection{Laws and applications}
Recall that a discrete group $\Gamma$ is virtually $P$ if it has a subgroup of finite index with $P$.
\begin{lemma}\label{lem:virtual}
Let $\mathcal{C} \subset \mathcal{RF}$ be a class of groups that is closed under subgroups.
If $P$ is a profinite property within $\mathcal{C}$, then being virtually $P$ is a profinite property within $\mathcal{C}$.
\end{lemma} 
\begin{proof}
Recall from Lemma \ref{lem:correspondence} that the finite index subgroups of $\Gamma$ are in correspondence with the open subgroups of $\widehat{\Gamma}$.
Let $\Gamma_1,\Gamma_2 \in \mathcal{C}$ and let  $\Psi \colon \widehat{\Gamma}_1 \to \widehat{\Gamma}_2$ be a continuous isomorphism.

Assume that $\Gamma_1$ is virtually $P$ and let $H_1 \leq_{\text{f.i.}} \Gamma_1$ be a finite index subgroup with $P$. 
Define $H_2 = \psi(\overline{H}_1)\cap \Gamma_2$. Then $\widehat{H}_2 \cong \overline{H}_2 \cong \overline{H}_1 \cong \widehat{H}_2$. 
By assumption $H_1, H_2$ lie in $\mathcal{C}$ and since $P$ is a profinite property within $\mathcal{C}$, we deduce that $H_2$ has $P$.
\end{proof}

Let $F_d$ be a finitely generated free group with generators $x_1,\dots,x_d$. We will refer to the elements of $F_d$ as $d$-letter words.  Let $G$ be a group. A word $w = w(x_1,\dots,x_d) \in F_d$ induces a \emph{word map} $w_G \colon G^d \to G$ by evaluating $(g_1,\dots,g_d) \mapsto w(g_1,\dots,g_n)$. If $G$ is a topological group, then the word map $w_G$ is continuous.

\begin{definition}
Let $G$ be a group. A word $w \in F_d$ is a \emph{law for $G$}, if 
$w_G(g_1,\dots,g_d) = 1_G$ for all $g_1,\dots,g_d$.
 \end{definition}
 For instance, a group $G$ is abelian exactly if $[x_1,x_2] = x_1x_2x_1^{-1}x_2^{-1}$ is a law for $G$.

\begin{proposition}\label{prop:law}
Let $w \in F_d$ and
let $\Gamma$ be residually finite group. Then $w$ is a law for $\Gamma$ if and only if it is a law for $\widehat{\Gamma}$.

\medskip

In particular, satisfying the law $w$ is a profinite property.
\end{proposition}
\begin{proof}
As $\Gamma$ is a subgroup of $\widehat{\Gamma}$ it is clear the if $w$ is a law for $\widehat{\Gamma}$, then it is a law for $\Gamma$.

For the converse direction we note that the density of $\Gamma$ in $\widehat{\Gamma}$ implies that $\Gamma^d$ is dense in $\widehat{\Gamma}^d$. The word map $w_{\widehat{\Gamma}}$ is continuous, so the preimage $w_{\widehat{\Gamma}}^{-1}(1)$ is closed. If $w$ is a law for $\Gamma$, then $w_{\widehat{\Gamma}}^{-1}(1)$ contains the dense set $\Gamma^d$ and hence agrees with~$\widehat{\Gamma}^d$.
\end{proof}
In particular, every property that can be defined in terms of group laws is profinite.
\begin{corollary}\label{cor}
Being abelian, solvable (of derived length $\ell$) or nilpotent (of class $c$) are profinite properties. \end{corollary}
Using Lemma \ref{lem:virtual} also being virtually abelian, virtually solvable or virtually nilpotent are profinite properties.
By Gromov's celebrated theorem \cite{Gromov81}, a finitely generated group is virtually nilpotent if and only if it has polynomial word growth. We deduce
\begin{corollary}\label{cor:word-growth-profinite}
Polynomial word growth is a profinite property.
\end{corollary}
The following problem suggests itself.
\begin{problem}
Prove Corollary \ref{cor:word-growth-profinite} without using Gromov's result.
\end{problem}
It is known from the work of Nekrashevych that in general the word growth of groups is not profinite \cite{Nekrashevych}.

\subsection{The first homology}
We have seen in Section \ref{sec:higherl2} that higher $\ell^2$-Betti numbers, and hence higher homology groups, are not profinite. However, the first homology is profinite and this has nice applications.
\begin{proposition}
Let $\Gamma \in \mathcal{RF}$.
The first homology $H_1(\Gamma,\Z)$ is profinite. In particular, the first Betti number $b_1(\Gamma,k) := \dim_k H^1(\Gamma,k)$ is a profinite invariant for every field $k$.
\end{proposition}
\begin{proof}
As the first homology of $\Gamma$ is merely the abelianisation, the assertion follows from Lemma~\ref{lem:profinite-commutes-ab} and the profinite rigidity of finitely generated abelian groups (see Example \ref{ex:abelian}).
\end{proof}

Following \cite[Def. 5.4]{Morales24} we say that a finitely generated residually finite group $\Gamma$ is \emph{L\"ucky}, if it satisfies the L\"uck approximation in degree~$1$. This means, for every descending chain $N_1 \supseteq N_2 \dots$ of finite index normal subgroups in $\Gamma$ with $\bigcap_{j\in \N} N_j = \{1\}$ we have
\[
	\lim_{j \to \infty} \frac{b_1(N_j,\Q)}{|\Gamma:N_j|} = b_1^{(2)}(\Gamma).
\]
By L\"uck's approximation theorem \cite{Lueck94} all finitely presented groups are L\"ucky and it is a folklore observation, that the proof works using the weaker homological finiteness property $\mathrm{FP}_2$.
It was shown in \cite{Morales24} that all finitely generated residually-(amenable and locally indicable) groups are L\"ucky. 

\begin{theorem} \label{thm:first-l2-betti-profinite}
The first $\ell^2$-Betti number is a profinite invariant among all L\"ucky groups.
\end{theorem}
\begin{proof}
Assume that $\Gamma, \Delta$ are L\"ucky and let $\Psi \colon \widehat{\Gamma} \to \widehat{\Delta}$ be an isomorphism of profinite groups.
Let $(N_j)_{j \in \N}$ be a descending chain of finite index normal subgroups of $\Gamma$ with $\bigcap_{j\in \N} N_j = \{1\}$. Then $\overline{N_j} \trianglelefteq_o \widehat{\Gamma}$ is an open normal subgroup. We consider the corresponding normal subgroup $M_j := \Psi(\overline{N_j}) \cap \Delta$. Observe that $|\Delta:M_j| = |\widehat{\Gamma}:\overline{N_j}| = |\Gamma:N_j|$. Since $\Delta$ is dense in $\widehat{\Delta}$ we find (using Lemma \ref{lem:correspondence}):
\[\widehat{M_j} \cong \overline{M_j} = \Psi(\overline{N_j}) \cong \overline{N_j} \cong \widehat{N_j}.\]
We deduce that $b_1(M_j,\Q) = b_1(N_j,\Q)$. Since $\Gamma,\Delta$ are L\"ucky, we deduce
\[
	b_1^{(2)}(\Gamma) = \lim_{j\to \infty} \frac{b_1(N_j,\Q)}{|\Gamma:N_j|} =  \lim_{j \to \infty} \frac{b_1(M_j,\Q)}{|\Delta:M_j|} = b_1^{(2)}(\Delta). \qedhere
\]
\end{proof}

 Non-L\"ucky groups exist, but the list of examples is rather short. L\"uck and Osin \cite{LueckOsin} constructed finitely generated, residually finite, infinite torsion groups with positive first $\ell^2$-Betti number; since the rational first Betti number of a torsion group vanishes, these groups cannot be L\"ucky. The torsion groups from Theorem~\ref{thm:hji} provide further examples. Both constructions provide groups with a lot of torsion homology; so in some sense the positive first $\ell^2$-Betti number could still be ``expected'' from looking at the profinite completion. In particular, the following problem is still open.
 \begin{problem}
Is the first $\ell^2$-Betti number a profinite invariant among all finitely generated residually finite groups?
\end{problem}

\subsection{Uniform amenability}
As mentioned above, it was proven in \cite{KS-amenable} that amenability is not profinite. In 1972 Keller introduced the stronger notion of uniform amenability by bounding the size of F{\o}lner sets:
\begin{definition}\label{def:uniform-foelner}
Let $G$ be a group.
We say that $G$ is $m$-\emph{uniformly amenable}
if for every $\varepsilon > 0$, and every finite subset $S \subseteq G$,
there exists a finite set $F \subseteq G$ satisfying
\begin{enumerate}
 \item $|F| \leq m(\varepsilon, |S|)$ and
 \item $|SF| \leq (1+\varepsilon)|F|$.
\end{enumerate}
for a fixed function $m\colon \R \times \N \to \N$. If such a function $m$ exists, we say that $G$ is \emph{uniformly amenable}.
\end{definition}
In a similar way, one can define uniformly amenable families of groups using a fixed function $m$; see \cite{KS-amenable}.
It was shown in \cite{KS-amenable} that a residually finite group is uniformly amenable if and only if the family of its finite quotients is uniformly amenable. From this one can deduce:
\begin{theorem}[Corollary 2.11 in \cite{KS-amenable}]\label{thm:uniformly-amenable}
Uniform amenability is a profinite property.
\end{theorem}
It follows from work of Kesten (see also \cite[Corollary 2.12]{KS-amenable}) that every uniformly amenable group satisfies a non-trivial law. If the answer to the following question is affirmative, then Theorem \ref{thm:uniformly-amenable} is equivalent to Proposition \ref{prop:law} above.
\begin{question}
Let $\Gamma$ be a finitely generated residually finite group. Assume that $\Gamma$ satisfies a non-trivial law. Is $\Gamma$ uniformly amenable?
\end{question}

\end{document}